\documentclass[12pt]{article}

\usepackage{amsmath,amsthm,amsfonts,amssymb,latexsym}
\usepackage{amsmath}

\usepackage{amssymb}
\newtheorem{thm}{Theorem}[section]

\newtheorem{lem}[thm]{Lemma}
\newtheorem{prop}[thm]{Proposition}

\newcommand{\bbZ}{\mathbb{Z}}

\newcommand{\Hom}{\mathrm{Hom}}

\newcommand{\fF}{{\mathfrak{F}}}

\newcommand{\fL}{{\mathfrak{L}}}

\newcommand{\cK}{{\mathcal{K}}}

\def\L{{\cal   L}}

\def\L1#1{L^1(#1)}

\def\lef({\left(}
\def\rig){\right)}

\begin{document}


\title{Second Cohomology  Space of the Lie
Superalgebra of contact vector fields $\mathcal{K}(2)$ with coefficients
 in the superspace of weighted densities on $S^{1|2}$ }

\author{ Othmen Ncib \thanks{D\'epartement de
Math\'ematiques, Facult\'e des Sciences de Gafsa, Zarroug 2112
Gafsa, Tunisie. E.mail: othmenncib@yahoo.fr} }

\date{}

\maketitle

\begin{abstract}
We investigate the second cohomology space of the Lie superalgebra $\mathcal{K}(2)$ with coefficients in the superspace of weighted densities on the (1,2)-dimentional real superspace. We explicitly give cocycles spanning this cohomology space.
\end{abstract}

\section{Introduction}
\vskip0.2cm

Let $\mathfrak{vect}(1)$ be the Lie algebra of vector fields on $S^1$. Consider the 1-parameter deformation of the $\mathfrak{vect}(1)$-action
on $\mathcal{C}^{\infty}(S^1)$:
$$L^{\lambda}_{X\frac{d}{dx}}(f)=Xf'+\lambda X'f,$$
where $X,\;f\in\mathcal{C}^{\infty}(S^1)$ and $X':=\frac{dX}{dx}$. This deformation shows that on the level of Lie algebras (and similary below, for Lie superalgebras) it is natural to choose $\mathbb{C}$ as the ground fields.\\
\vskip0.2cm Denote by $\mathcal{F}_{\lambda}$ the $\mathfrak{vect}(1)$-module structure on $\mathcal{C}^{\infty}(S^1)$ denoted by $L^{\lambda}$ for a fixed $\lambda$. Geometrically, $\mathcal{F}_{\lambda}=\{fdx^{\lambda}/f\in\mathcal{C}^{\infty}(S^1)\}$ is the space of weighted densities of weighted $\lambda\in\mathbb{K}$. The space $\mathcal{F}_{\lambda}$ coincides with the space of vector fields, fuctions and differential 1-form for $\lambda=-1,\;0$ and 1, respectively.\\
\vskip0.2cm Now, we consider the superspace $S^{1|n}$ equipped with its standard contact structure 1-form $\alpha_n=dx+\displaystyle\sum_{i=1}^n\theta_id\theta_i$, and introduce the superspace $\mathfrak{F}^n_{\lambda}$ of $\lambda$-densities on the supespace $S^{1|n}$. Let $\mathcal{K}(n)$ be the Lie superalgebra of contact vector fields, $\mathfrak{F}^n_{\lambda}$ in naturally a $\mathcal{K}(n)$-module.\\
\vskip0.2cm  The study of $H^2(Vect(S^1),\mathcal{F}_{\lambda})$ is given in \cite{Fuchs}, and the super corresponding case $H^2(\mathcal{K}(1),\mathfrak{F}^1_{\lambda})$ is given by B. Agrebaoui, I. Basdouri and M. Boujelbene in \cite{bim}  . In this paper, we compute the second cohomology space of the $\mathcal{K}(2)$-module of tensor densities:
$$\mathfrak{F}^2_{\lambda}=\{F\alpha_2^{\lambda}\;:\;F\in\mathcal{C}^{\infty}(S^{1|2})\},$$
where the action of $\mathcal{K}(2)$ on $\mathfrak{F}^2_{\lambda}$ is given by the Lie derivatives:
$$\mathfrak{L}^{\lambda}_{v_F}(\Phi)=(v_F(G)+\lambda F'G)\alpha_2^{\lambda}$$
where $F,\;G\in\mathcal{C}^{\infty}(S^{1|2}),\;\Phi=G\alpha^{\lambda}_2$ and $F'=\partial_xF$. This space classify central extensions of
$\mathcal{K}(2)$ by tensor densities $\mathfrak{F}^2_{\lambda}$. \\

\section{Geometry of the superspace $S^{1|2}$}

\subsection{The Lie superalgebra of contact vector fields on $S^{1|2}$}

Let $S^{1\mid 2}$ be the superspace with local
coordinates $(x;\;\theta_1,\theta_2),$ where
$\theta=(\theta_1,\theta_2)$ are the odd variables. Any
contact structure on $S^{1\mid 2}$ can be given by the
following $1$-form:
\begin{equation}
\label{a} \alpha_n=dx+\sum_{i=1}^2\theta_id\theta_i.
\end{equation}

On the space
$C^{\infty}(S^{1|2})$, we
consider the contact bracket
\begin{equation}
\{F,G\}=FG'-F'G-\frac{1}{2}(-1)^{|F|}\sum_{i=1}^2{\overline{\eta}}_i(F)\cdot
{\overline{\eta}}_i(G),
\end{equation} where the superscript ' stands for $\frac{\partial}{\partial x},$
$\overline{\eta}_i=\frac{\partial}{\partial
{\theta_i}}-\theta_i\frac{\partial}{\partial x}$ and $|F|$ is the
parity of $F$. Note that the derivations $\overline{\eta}_i$ are
the generators of $2$-extended supersymmetry and generate the kernel
of the form (\ref{a}) as a module over the ring of smooth
functions. Let $\mathrm{Vect}(S^{1|2})$ be the
superspace of vector fields on $S^{1|2}$:
\begin{equation*}\mathrm{Vect}(S^{1|2})=\left\{F_0\partial_x
+ \sum_{i=1}^2 F_i\partial_i \mid ~F_i\in C^{\infty}(S^{1|2})~,
 i=1,2  \right\},\end{equation*}  and consider the
superspace $\mathcal{K}(2)$ of contact vector fields on
$S^{1|2}$. That is, $\mathcal{K}(2)$ is the superspace
of vector fields on $S^{1|2}$ preserving the distribution
singled out by the $1$-form $\alpha_2$: $$
\mathcal{K}(2)=\left\{X\in\mathrm{Vect}(S^{1|2})~|~\hbox{there
exists}~F\in C^{\infty}(S^{1|2})~ \hbox{such
that}~{L}_X(\alpha_2)=F\alpha_2\right\}. $$
where $L_X$ is the Lie derivative along the vector field $X.$

 The Lie superalgebra
$\mathcal{K}(2)$ is spanned by the fields of the form:
\begin{equation*}
X_F=F\partial_x
-\frac{1}{2}\sum_{i=1}^2(-1)^{|F|}{\overline{\eta}}_i(F){\overline{\eta}}_i,\;\text{where
$F\in C^{\infty}(S^{1|2})$.}
\end{equation*}
 The bracket in
$\mathcal{K}(2)$ can be written as:
\begin{equation*}
[X_F,\,X_G]=X_{\{F,\,G\}}.
\end{equation*}

For every contact vector field $X_F$, one define a
one-parameter family of first-order differential operators on
$C^{\infty}(S^{1|2})$:
\begin{equation}
\label{superaction} \fL^{\lambda}_{X_F}=X_F + \lambda F',\quad
\lambda\in \mathbb{K}.
\end{equation}
We easily check that
\begin{equation}
\label{crochet}
[\fL^{\lambda}_{X_F},\fL^{\lambda}_{X_G}]=\fL^{\lambda}_{X_{\{F,G\}}}.
\end{equation}
We thus obtain a one-parameter family of $\cK(2)$-modules on
$C^{\infty}(S^{1|2})$ that we denote $\mathfrak{F}^2_\lambda$,
the space of all weighted densities on
$C^{\infty}(S^{1|2})$ of weight $\lambda$ with respect to $\alpha_2$:
\begin{equation}
\label{densities} \fF^2_{\lambda}=\left\{F\alpha_2^{\lambda} \mid F
\in C^{\infty}(S^{1|2})\right\}.
\end{equation}
In particular, we have $\mathfrak
{F}_{\lambda}^0=\mathcal{F}_\lambda$.
 Obviously the adjoint
$\cK(2)$-module is isomorphic to the space of weighted densities
on $C^{\infty}(S^{1|2})$ of weight $-1.$
The case $n = 0$ corresponds to the classical
setting: $\cK(0)=\mathrm{Vect}(S^1)= \{ F\partial_x | F\in C^{\infty}(S^1) \}$.
Note that, the Lie superalgebra $\mathcal{K}(1)$ is isomorphic to  $$\mathcal{K}(1)_i =\{ X_F \in \mathcal{K}(2)~
 | ~\partial_iF=0,\;i=1,2\}.$$
 Therefore, the spaces of weighted densities $\mathfrak{F}_{\lambda}^n$ are also
 $\mathcal{K}(n-1)$-modules.
 In \cite{{imnm},{hnv}} it was proved that, as $\mathcal{K}(1)$-module, we have
\begin{equation}\label{ot1} \mathfrak{F}_{\lambda}^2\simeq
{\mathfrak{F}}_{\lambda}\oplus \Pi
(\mathfrak
{\mathfrak{F}}_{\lambda+\frac{1}{2}}).
\end{equation}
where $\Pi$ is the change of parity operator.

\subsection{Lie superalgebra cohomology }
\vskip 0.2 cm
Let $\frak{g}=\frak{g}_0\oplus \frak{g}_1$
be a Lie superalgebra acting on a super vector space $V=V_0\oplus
V_1$. The space of $k$-cochaines with values in $V$ is the $\frak{g}$-module
 $$C^k(\frak{g},V):= \Hom(\wedge^k\frak{g};\, V).$$
 The {\it coboundary operator} $\delta_k:C^k(\frak{g},V)\longrightarrow C^{k+1}(\frak{g},V)$ is a $\frak{g}$-map
  satisfying $\delta_k\circ \delta_{k-1}=0.$ The kernel of $\delta_k,$ denoted $\hbox{Z}^k(\frak{g},V),$ is the space of $k$-cocycles, among them, the elements in the range of $\delta_{k-1}$ are called {\it $k$-coboundries.}
 We denote $\hbox{B}^k(\frak{g},V)$ the space of $k$-coboundries.

 By definition, the $k^{th}$ cohomology space is the quotient space
 $$\hbox{H}^k(\frak{g},V)=\hbox{Z}^k(\frak{g},V)/\hbox{B}^k(\frak{g},V).$$
We will only need the formula of $\delta_n$ (which will be simply denoted $\delta$) in degrees 0, 1 and
2: for $v\in C^0(\frak{g},V)=V,\, \delta v(x):= (-1)^{|x||v|}x.v,$ for $\Upsilon \in C^1(\frak{g},V),$
\begin{equation}
\label{}
\delta(\Upsilon)(x,y):=(-1)^{|x||\Upsilon|}x.\Upsilon(y)-(-1)^{|y|(|x|+|\Upsilon|)}y.\Upsilon(x)-\Upsilon([x,\,y]),
\end{equation}
and for $\Omega \in C^2(\frak{g},V),$
\begin{equation}
\label{}
\begin{array}{lllll}
\delta(\Omega)(x,y,z):=& (-1)^{|\Omega||x|}x.\Omega(y,z)-(-1)^{|y|(|\Omega|+|x|)}y.\Omega(x,z) +(-1)^{|z|(|\Omega|+|x|+|y|)}z.\Omega(x,y)\\[5pt]
&-\Omega([x,y],z)
+(-1)^{|y||z|}\Omega([x,z],y)-(-1)^{|x|(|y|+|z|)}\Omega([y,z],x),
\end{array}
\end{equation}
where $x,y,z\in \frak{g}.$

 $\Hom(\frak{g},\, V)$ is $\bbZ_2$-graded via
\begin{equation}
\label{grade} \Hom(\frak{g}, V)_b=\displaystyle \oplus_{a\in
\bbZ_2}\Hom(\frak{g}_a, V_{a+b}); \; b\in \bbZ_2.
\end{equation}
\hfill $\Box$
\vskip 0.2 cm In this paper, we study the differential cohomology spaces $H^2(\mathcal{K}(2),\mathfrak{F}^2_{\lambda})$. That is, we consider only cochains $(X_F,X_G)\mapsto\Omega(F_1,F_2)\alpha_2^{\lambda}$ where $\Omega$ is a differential operator.

\section{The space $H^2(\mathcal{K}(2),\mathfrak{F}^2_{\lambda})$}
\vskip0.2cm
Recall that the Lie superalgebra $\mathcal{K}(2)$ has two subsuperalgebras
$\mathcal{K}(1)_i \hbox{ for }~i=1,~2$ isomorphic to $\mathcal{K}(1)$ defined by
\begin{equation}\label{ot2}
\mathcal{K}(1)_i=\Big\{ X_{F}~|~~F=f_0+f_i\theta_i, \hbox{ where }
f_0,~f_i \in C^{\infty}(S^1)\Big\}.
\end{equation}

For $i=1,~2,$ let $\Im_{\lambda}^{1,i}$ be the $\mathcal{K}(1)_i$-module of
tensor densities of degree $\lambda$ on~$S_i^{1\mid 1},$ where
$S_i^{1\mid 1}$ is the superline with local coordinates
$(x,\theta_i).$\\
The main result in this paper is the following:\\
\begin{thm}\label{ot12}
\begin{equation}\label{ot13}
\hbox{H}_{diff}^2(\mathcal{K}(2),\mathfrak{F}_{\lambda}^2)\simeq\left\{
\begin{array}{ll}
\mathbb{K}^3&\makebox{ if }~\lambda=0,\\\mathbb{K}^2&\makebox{ if }~\lambda=1, \\ \mathbb{K} &\makebox{ if }~\lambda=3,\\
0&\makebox { otherwise }.
\end{array}
\right.\end{equation}
\vskip 0.2 cm The non trivial spaces $\hbox{H}_{diff}^2(\mathcal{K}(2),\mathfrak{F}_{\lambda}^2)$ are spanned by the following 2-cocycles:
\begin{equation}
\label{H2013}
\begin{array}{ll}
\bullet\; \displaystyle \Upsilon^0_1(X_F,X_G) =\displaystyle
(FG'-F'G)\\-\Big(\frac{1}{4}+\frac{3}{4}(-1)^{|F||G|}\Big)(-1)^{|F|(|G|+1)}\Big(\overline{\eta}_1(F)\overline{\eta}_1(G)+
\overline{\eta}_2(F)\overline{\eta}_2(G)\Big)
\\+\frac{1}{2}\Big((-1)^{|F|}\overline{\eta}_1\overline{\eta}_2(F)\overline{\eta}_2(G)-
\overline{\eta}_2(F)\overline{\eta}_1{\overline\eta}_2(G)\Big)\theta_1\\
+\frac{1}{2}\Big(\overline{\eta}_1(F)\overline{\eta}_1\overline{\eta}_2(G)-
(-1)^{|F|}\overline{\eta}_1\overline{\eta}_2(F)\overline{\eta}_1(G)\Big)\theta_2,\\
\\[5pt]
\bullet\;\displaystyle \Upsilon^0_2(X_F,X_G) =\displaystyle \Big(\overline{\eta}_1\overline{\eta}_2(F)G-F\overline{\eta}_1\overline{\eta}_2(G)\Big)\\
+\frac{1}{2}\Big((-1)^{|G|}\overline{\eta}_1(F)\overline{\eta}_1\overline{\eta}_2(G)-\overline{\eta}_1\overline{\eta}_2(F)\overline{\eta}_1(G)\Big)\theta_1\\
+\frac{1}{2}\Big((-1)^{|G|}\overline{\eta}_2(F)\overline{\eta}_1\overline{\eta}_2(G)-\overline{\eta}_1\overline{\eta}_2(F)\overline{\eta}_2(G)\Big)\theta_2\\
+2\Big(FG''-F''G\Big)\theta_1\theta_2\\
\\[5pt]
\bullet\;\displaystyle \Upsilon^0_3(X_F,X_G) =\displaystyle\Big(F'\overline{\eta}_1\overline{\eta}_2(G)-\overline{\eta}_1\overline{\eta}_2(F)G'\Big).
&\hbox{ if }
\lambda=0,\\
\\[5pt]
\bullet\;\displaystyle \Upsilon^1_0(X_F,X_G)
=\displaystyle \Big((-1)^{|F|}F'\overline{\eta}_1\overline{\eta}_2(G')-(-1)^{|G|}\overline{\eta}_1\overline{\eta}_2(F')G'\Big)\alpha_2\\
\\[5pt]
\bullet\;\displaystyle \Upsilon^1_1(X_F,X_G)=\displaystyle (-1)^{(|F|+|G|)}\Big(\overline{\eta}_1\overline{\eta}_2(F)\overline{\eta}_1\overline{\eta}_2(G')
-\overline{\eta}_1\overline{\eta}_2(F')\overline{\eta}_1\overline{\eta}_2(G)\Big)\alpha_2\\

 &\hbox{ if } \lambda=1,\\
\\[5pt]
\bullet\;\displaystyle \Upsilon^3_0(X_F,X_G)=\displaystyle \Big[(-1)^{|F|+1}\Big(\overline{\eta}_1(F'')\overline{\eta}_1(G'')+\overline{\eta}_2(F'')\overline{\eta}_2(G'')\Big)\\
+2\Big(\overline{\eta}_1\overline{\eta}_2(F'')\overline{\eta}_1\overline{\eta}_2(G')-
\overline{\eta}_1\overline{\eta}_2(F')\overline{\eta}_1\overline{\eta}_2(G'')\Big)
\\+6\Big(\overline{\eta}_1(F'')\overline{\eta}_2(G^{(3)})-\overline{\eta}_2(F^{(3)})\overline{\eta}_1(G'')\Big)\Big]\alpha_2^3
 &\hbox{ if } \lambda=3.
\\
\medskip
\end{array}
\end{equation}
\end{thm}

\subsection{Relationship between $H^2_{diff}(\mathcal{K}(2),\mathfrak{F}^2_{\lambda})$ and $H^2_{diff}(\mathcal{K}(1),\mathfrak{F}^2_{\lambda})$}
\vskip0.2cm

According to (\ref{ot1}), we can see that the second cohomology space $H^2_{diff}(\mathcal{K}(1),\mathfrak{F}^2_{\lambda})$ is closely related to the space
$H^2_{diff}(\mathcal{K}(1),\mathfrak{F}^1_{\lambda})$:
\begin{equation}\label{ot19}
H^2_{diff}(\mathcal{K}(1),\mathfrak{F}^2_{\lambda})=H^2_{diff}(\mathcal{K}(1),\mathfrak{F}^1_{\lambda})\oplus
H^2_{diff}(\mathcal{K}(1),\Pi(\mathfrak{F}^1_{\lambda+\frac{1}{2}})).
\end{equation}
Thus, we can easely deduce the following result:\\

\begin{prop}
\label{ot14} 1) The cohomology space
\begin{equation*}
H^2(\mathcal{K}(1)_i,\mathfrak{F}_{\lambda}^2)_0=\left\{
\begin{array}{ll}
\mathbb{K}^3&\makebox{ if }~\lambda=0, \\
\mathbb{K}^2&\makebox{ if }~\lambda=1, \\
\mathbb{K}&\makebox{ if }~\lambda=3,5, \\
0&\makebox { otherwise }.
\end{array}
\right.\end{equation*} The respective nontrivial $2$-cocycles are
\begin{equation}
\label{C012}
\begin{array}{ll}
\bullet\;\displaystyle \Omega_{0,1}^i(X_F,X_F) =\displaystyle
(FG'-F'G)-\Big( \frac{1}{4}+\frac{3}{4}(-1)^{|F||G|} \Big)\overline{\eta}_i(F)\eta_i(G) ,\\
\\[5pt]
\bullet\;\displaystyle \Omega_{0,2}^i(X_F,X_G)
=\displaystyle (-1)^{|F|+|G|}\Big(F'\eta_i(G')-\eta_i(F')G'\Big)\theta_{3-i},\\
\\[5pt]
\bullet\;\displaystyle
\Omega_{0,3}^i(X_F,X_G)=\Big( \frac{1}{2}+\frac{1}{4}(-1)^{|F||G|} \Big)(-1)^{|F|+|G|}\Big(F\eta_i(G')-\eta_i(F')G\Big)\theta_{3-i},\\
\\[5pt]
\bullet\;\displaystyle
\Omega_{1,1}^i(X_F,X_G)=\Big(\overline{\eta}_i(F'')G-(-1)^{|F|}F\overline{\eta}_i(G'') \Big)\theta_{3-i}-
\frac{1}{2}\Big(\eta_i(F)\eta_i(G'')+\eta_i(F'')\eta_i(G)\Big)\theta_i\theta_{3-i}\alpha_2,\\
\\[5pt]
\bullet\;\displaystyle
\Omega_{1,2}^i(X_F,X_G)=\Big(F'\overline{\eta}_i(G'')-\overline{\eta}_i(F'')G'\Big)\theta_{3-i}\alpha_2,\\
\\[5pt]
\bullet\;\displaystyle
\Omega_{3}^i(X_F,X_G)=\eta_i(F'')\overline{\eta}_i(G'')\alpha_2^3 ,\\
\\[5pt]
\bullet\;\displaystyle
\Omega_{5}^i(X_F,X_G)=\Big( (F^{(3)}G^{(4)}-(F^{(4)}G^{(3)})+\frac{3}{2}(\eta_i(F^{(4)})\eta_i(G^{(2)})\\-
(-1)^{|F||G|}\eta_i(F^{(2)})\eta_i(G^{(4)})-4\eta_i(F^{(3)})\eta_i(G^{(3)}))\Big)\alpha_2^5.
\end{array}
\end{equation}
where $X_F,~X_G\in \mathcal{K}_i,~~i=1,~2.$\\
\medskip
2) The cohomology space
\begin{equation*}
\hbox{H}^2(\mathcal{K}(1)_i,\mathfrak{F}_{\lambda}^2)_1=\left\{
\begin{array}{ll}
\mathbb{K}^2&\makebox{ if }~\lambda=\frac{1}{2},\frac{3}{2}\\ \mathbb{K} &\makebox{ if }~\lambda=-\frac{1}{2},\frac{5}{2},\frac{9}{2} \\
0&\makebox { otherwise }.
\end{array}
\right.\end{equation*} The respective nontrivial $2$-cocycles are
\begin{equation}
\label{C0123}
\begin{array}{llll}
\bullet\;\displaystyle \Omega_{-\frac{1}{2}}^i(X_F,X_G) =\displaystyle
\Big[(FG'-F'G)\theta_{3-i}-\Big( \frac{1}{4}+\frac{3}{4}(-1)^{|F||G|} \Big)\overline{\eta}_i(F)\eta_i(G)\theta_{3-i}\Big]\alpha_2^{-\frac{1}{2}}~~
~~~~~\hbox{ if }
\lambda=-\frac{1}{2},\\
\\[2pt]
\bullet\;\left\{
  \begin{array}{ll}
   \displaystyle\Omega_{\frac{1}{2}}^i(X_F,X_G)
=\displaystyle \Big[(-1)^{|F|+|G|}\Big( F'\eta_i(G')-\eta_i(F')G' \Big) \Big]\alpha_2^{\frac{1}{2}},&  \\[5pt]
    \widetilde{\Omega}_{\frac{1}{2}}^i(X_F,X_G)= \Big[ \Big( \frac{1}{2}+\frac{1}{4}\big( 1+(-1)^{|F||G|} \big) \Big) (-1)^{|F|+|G|}
    \Big(F\eta_i(G')-\eta_i(F')G\Big)\Big]\alpha_2^{\frac{1}{2}}, ~~~~ \hbox{ if } \lambda=\frac{1}{2},
  \end{array}
\right.\\
\\[2pt]
\bullet\;\left\{
  \begin{array}{ll}
    \displaystyle\Omega_{\frac{3}{2}}^i(X_F,X_G)
=\displaystyle \Big(\overline{\eta}_i(F'')G-(-1)^{|F|}F\overline{\eta}_i(G'') \Big)-
\frac{1}{2}\Big(\eta_i(F)\eta_i(G'')+\eta_i(F'')\eta_i(G)\Big)\theta_i\alpha_2^{\frac{3}{2}},&  \\[5pt]
    \widetilde{\Omega}_{\frac{3}{2}}^i(X_F,X_G)=\Big(F'\overline{\eta}_i(G'')-\overline{\eta}_i(F'')G'\Big) \alpha_2^{\frac{3}{2}}, ~ \hbox{ if } \lambda=\frac{3}{2},\\
 \end{array}
\right.
\\[2pt]
\bullet\;\displaystyle \Omega_{\frac{5}{2}}^i(X_F,X_G) =\displaystyle
\eta_i(F'')\overline{\eta}_i(G'')\theta_i\alpha_2^{\frac{5}{2}}~~
~~~~~~~~~~~~~~~~~~~~~~~~~~~~~~~~~~~~~~~~~~~~~~~~~~~~~~~~~~~~\hbox{ if }
\lambda=\frac{5}{2},\\
\\[2pt]
\bullet\;\displaystyle \Omega_{\frac{9}{2}}^i(X_F,X_G) =\displaystyle
\Big( (F^{(3)}G^{(4)}-(F^{(4)}G^{(3)})+\frac{3}{2}(\eta_i(F^{(4)})\eta_i(G^{(2)})\\-
(-1)^{|F||G|}\eta_i(F^{(2)})\eta_i(G^{(4)})-4\eta_i(F^{(3)})\eta_i(G^{(3)}))\Big)\theta_i\alpha_2^{\frac{9}{2}}.~~
~~~~~~~~~~~~~~~~~~~~~~~~~~~~~~~~~\hbox{ if }
\lambda=\frac{9}{2},\\
\end{array}
\end{equation}
where $X_F,~X_G\in \mathcal{K}(1)_i,~~i=1,~2.$
\end{prop}

To prove proposition \ref{ot14}, we need the following result (see \cite{bim} ).

\begin{prop}\cite{bim}
\label{cohomdensities11}
The cohomology space
\begin{equation*}
\hbox{H}^2(\mathcal{K}(1)_i,\Im_{\lambda}^{1,i})=\left\{
\begin{array}{ll}
\mathbb{K}^2&\makebox{ if }~\lambda=\frac{1}{2},\;\frac{3}{2}\\ \mathbb{K} &\makebox{ if }~\lambda=0,\;3,\;5 \\
0&\makebox { otherwise }.
\end{array}
\right.\end{equation*} The respective nontrivial $2$-cocycles are
\begin{equation}
\label{C01232}
\begin{array}{llll}
\bullet\;\displaystyle \Omega_0^i(X_F,X_G) =\displaystyle
(FG'-F'G)-\big( \frac{1}{4}+\frac{3}{4}(-1)^{|F||G|} \big)\overline{\eta}_i(F)\eta_i(G) ~~~~~~~~~~~~~~
~~~~~~~~~~~~~~\hbox{ if }
\lambda=0,\\
\\[2pt]
\bullet\;\left\{
  \begin{array}{ll}
 \displaystyle  \Omega_{\frac{1}{2}}^i(X_F,X_G)
=\displaystyle \Big[(-1)^{|F|+|G|}\Big( F'\eta_i(G')-\eta_i(F')G' \Big) \Big]\alpha_{1,i}^{\frac{1}{2}},&  \\[5pt]
    \widetilde{\Omega}_{\frac{1}{2}}^i(X_F,X_G)= \Big[ \Big( \frac{1}{2}+\frac{1}{4}\big( 1+(-1)^{|F||G|} \big) \Big)
    (-1)^{|F|+|G|}\Big(F\eta_i(G')-\eta_i(F')G\Big)\Big]\alpha_{1,i}^{\frac{1}{2}}, ~ \hbox{ if } \lambda=\frac{1}{2},
  \end{array}
\right.
\\[2pt]
\bullet\;\left\{
  \begin{array}{ll}
  \displaystyle  \Omega_{\frac{3}{2}}^i(X_F,X_G)
=\displaystyle \Big(\overline{\eta}_i(F'')G-(-1)^{p(F)}F\overline{\eta}_i(G'')\Big)-\frac{1}{2}\Big(\eta(F)\eta(G'')+\eta(F'')\eta(G)\Big)
\theta_i\alpha^{\frac{3}{2}}_{1,i},&  \\[5pt]
    \widetilde{\Omega}_{\frac{3}{2}}^i(X_F,X_G)=\Big(F'\overline{\eta}_i(G'')-\overline{\eta}_i(F'')G'\Big) \alpha_{1,i}^{\frac{3}{2}}, ~ \hbox{ if } \lambda=\frac{3}{2},\\
  \end{array}
\right.\\
\\[2pt]
\bullet\;\displaystyle\Omega_3(X_F,X_G)=\eta_i(F'')\overline{\eta}_i(G'')\alpha_{1,i}^3~~~~~~~~~~~~~~~~~~~~~~~~~~~~~~~~~~~~~~~~~~~~~~~~~
~~~~~~~~~~~\hbox{ if }\lambda=3\\
\\[2pt]
\bullet\;\displaystyle\Omega_5(X_F,X_G)=\Big( (F^{(3)}G^{(4)}-(F^{(4)}G^{(3)})+\frac{3}{2}(\eta_i(F^{(4)})\eta_i(G^{(2)})\\-
(-1)^{|F||G|}\eta_i(F^{(2)})\eta_i(G^{(4)})-4\eta_i(F^{(3)})\eta_i(G^{(3)}))\Big)\alpha_{1,i}^5.
\end{array}
\end{equation}
where $X_F,~X_G\in \mathcal{K}(1)_i~\hbox{and}~\alpha_{1,i}=dx+\theta_i d\theta_i,~i=1,~2.$
\end{prop}

\noindent{\it Proof of Proposition \ref{ot14}}: Let
$F\alpha_2^{\lambda}=(f_0+f_1\theta_1+
f_2\theta_2+f_{12}\theta_1\theta_2)\alpha_2^{\lambda}
\in\fF_{\lambda}^2$. The map
\begin{equation*}
\begin{array}{lcll} \Phi:&\fF_{\lambda}^2 &\longrightarrow&
\Im_{\lambda}^{1,i}\oplus \Pi(\Im_{\lambda+\frac{1}{2}}^{1,i})\\
&F\alpha_2^{\lambda}&\longmapsto&\Big((1-\theta_{3-i}\partial_{{\theta}_{3-i}})(F)
\alpha_{1,i}^{\lambda},~
\Pi\big((-1)^{|F|+1}\partial_{\theta_{3-i}}(F)\alpha_{1,i}^{\lambda+\frac{1}{2}}\big)\Big),
\end{array}
\end{equation*}
where $i=1, 2$ and $\Pi$ stands for the parity change map,
 provides us with an isomorphism of $\mathcal{K}(1)_i$-modules. In fact, we easily check  that
$$\mathfrak{L}_{X_H}^\lambda(F)\alpha_2^\lambda=
\Big(\mathfrak{L}_{X_H}^\lambda \big((1-\theta_{3-i}\partial_{{\theta}_{3-i}})(F)\big)+
\mathfrak{L}_{X_H}^{\lambda+\frac{1}{2}}\big((-1)^{|F|+1}\partial_{\theta_{3-i}}(F)\big)\theta_{3-i}\Big)
\alpha_2^\lambda.$$

 This map induces the
following isomorphism between cohomology spaces:
\begin{equation*}
\hbox{H}^2(\mathcal{K}(1)_i,~\fF_{\lambda}^2) \cong \hbox{H}^2(\mathcal{K}(1)_i,~\Im^{1,i}_{\lambda})
\oplus \hbox{H}^2(\mathcal{K}(1)_i,~\Pi(\Im^{1,i}_{\lambda+\frac{1}{2}})).
\end{equation*}
We deduce from this isomorphism and Proposition \ref{cohomdensities11}, the
$2$-cocycles (\ref{C012}--\ref{C0123}). \hfill $\Box$\\
\vskip 0.2 cm
Recall that the adjoint $\mathcal{K}(n)$-module, is isomorphic to $\mathfrak{F}^n_{-1}.$ Thus, the $\mathcal{K}(1)_i$-isomorphism (\ref{ot1}) yields the following $\mathcal{K}(1)_i$-isomorphism:
\begin{equation}\label{ot20}
\mathcal{K}(2)\simeq\mathcal{K}(1)\oplus\Pi(\mathcal{H}^i),
\end{equation}
where $\mathcal{H}^i$ is isomorphic to $\mathfrak{F}^{1,i}_{-\frac{1}{2}}$. More precisely, any element $X_F\in\mathcal{K}(2)$ is decomposed into
$X_F=X_{F_1}+X_{F_2\theta_{3-i}}$, where $\partial_{\theta_{3-i}}F_1=\partial_{\theta_{3-i}}F_2=0$, and then $X_{F_1}\in\mathcal{K}(1)$ and $X_{F_2\theta_2}$ is identified to
$\Pi(F_2\alpha_{1,i}^{-\frac{1}{2}})\in\Pi(\mathcal{H}^i)$ and it will be denoted $X_{\overline{F}_2}$. Moreover, we can easily that:
\vskip 0.2 cm
\begin{equation}\label{ot21}
[\mathcal{K}(1),\Pi(\mathcal{H}^i)]\subset\Pi(\mathcal{H}^i)\; and\; [\Pi(\mathcal{H}^i),\Pi(\mathcal{H}^i)]\subset\mathcal{K}(1)
\end{equation}

The following lemma gives the general form of each 2-cocycle.\\
\vskip 0.2 cm
\begin{lem}\label{ot15}
Let $\Omega\in Z^2_{diff}(\mathcal{K}(2),\mathfrak{F}^2_{\lambda})$. Up to a coboundary, the map $\Omega$ are given by
$$\Omega(X_F,X_G)=\displaystyle\sum_{i,j,k,l\geq0}a_{ij}^{kl}\overline{\eta}_1^i\overline{\eta}_2^j(F)
\overline{\eta}_1^k\overline{\eta}_2^l(G)\alpha^{\lambda},\;where\;\frac{\partial}{\partial_x}a_{ij}^{kl}=0.$$

\end{lem}
Proof. Evry differential operator $\Omega$ can be expressed in the form \\ $\Omega(X_F,X_G)=\displaystyle\sum_{i,j,k,l\geq0}a_{ij}^{kl}\overline{\eta}_1^i\overline{\eta}_2^j(F)
\overline{\eta}_1^k\overline{\eta}_2^l(G)\alpha^{\lambda},\;where\;\frac{\partial}{\partial_x}a_{ij}^{kl}=0$, where the coefficients $a_{ij}^{kl}$
are arbitrary functions. Using the 2-cocycle equation, we show that $\frac{\partial}{\partial_x}a_{ij}^{kl}=0$. That is, the coefficients
$a_{ij}^{kl}$ are not depending on the variable $x$, but they are depending on $\theta_i,\;i=1,2$ and on the parity of $F$ and $G$. The dependence on the parity of $F$ and $G$ becomes from the fact that $\Omega$ is skew-symmetric:
$$a_{ij}^{kl}(F,G)=(-1)^{\varepsilon_{ij}^{kl}(F,G)}a_{kl}^{ij}(F,G)\;where\;\varepsilon_{ij}^{kl}(F,G)=ijkl(p(F)+1)(p(G)+1))+p(F)p(G)+1,$$
indeed, $\overline{\eta}_i$ is an odd operator.
\vskip 0.2 cm The main result in this subsection is the following:
\begin{prop}\label{ot16}
There exist, up to a scalar factor and a coboundary, only three non trivial 2-cocycles $\Omega_0,\;\Omega_1$ and $\widetilde{\Omega}_1$
from $\mathcal{K}(2)$ to $\mathfrak{F}^2_{\lambda}$, given by

\begin{equation}\label{ot18}
\begin{array}{ll}
\displaystyle \Omega_0(X_F,X_G) =\displaystyle\Big(F'\overline{\eta}_1\overline{\eta}_2(G)-\overline{\eta}_1\overline{\eta}_2(F)G'\Big)\\
 &\hbox{ if } \lambda=0,\\
\\[5pt] \displaystyle \Omega_1(X_F,X_G)
=\displaystyle \Big((-1)^{|F|}F'\overline{\eta}_1\overline{\eta}_2(G')-(-1)^{|G|}\overline{\eta}_1\overline{\eta}_2(F')G'\Big)\alpha_2\\
\\[5pt]
\displaystyle \widetilde{\Omega}_1(X_F,X_G)=\displaystyle (-1)^{(|F|+|G|)}\Big(\overline{\eta}_1\overline{\eta}_2(F)\overline{\eta}_1\overline{\eta}_2(G')
-\overline{\eta}_1\overline{\eta}_2(F')\overline{\eta}_1\overline{\eta}_2(G)\Big)\alpha_2,\\
&\hbox{ if } \lambda=1.\\
\end{array}
\end{equation}
such that any nonzero linear combination is a nontrivial 2-cocycle and their restrictions to $\mathcal{K}(1)_1$ or $\mathcal{K}(1)_2$ are coboundaries.\\

\end{prop}

\begin{proof}
 Let $\Upsilon$ a 2-cocycle of $\mathcal{K}(2)$ vanishing on $\mathcal{K}(1)_i,\;i=1,2$ (for example $\mathcal{K}(1)_1$).\\
Assume that $\Upsilon/\mathcal{K}(1)_1$ is a coboundary, that is, there exist $\Phi\in\mathfrak{F}^2_{\lambda}$ such that
$$\Upsilon(X_{F_1},X_{G_1})=\delta(\Phi)(X_{F_1},X_{G_1})\;forall\;X_{F_1},\;X_{G_1}\in\mathcal{K}(1)_1.$$
By replacing $\Upsilon$ by $\Upsilon-\delta(\Phi)$, we can suppose that $\Upsilon$ vanishes on $\mathcal{K}(1)_1.$ But in this case according to
(\ref{ot21}), the 2-cocycle relations read
\begin{equation}\label{ot22}
\begin{array}{ll}
\bullet\; \displaystyle(-1)^{|F_1||\Upsilon|}\mathfrak{L}^{\lambda}_{X_{F_1}}(\Upsilon(X_{G_1},X_{H\theta_2}))-
(-1)^{|G_1|(|F_1|+|\Upsilon|)}\mathfrak{L}^{\lambda}_{X_{G_1}}(\Upsilon(X_{F_1},X_{H\theta_2}))-\Upsilon([X_{F_1},X_{G_1}],X_{H\theta_2})\\+
(-1)^{|G_1|(|H|+1)}\Upsilon([X_{F_1},X_{H\theta_2}],X_{G_1})-(-1)^{|F_1|(|G_1|+|H|+1)}\Upsilon([X_{G_1},X_{H\theta_2}],X_{F_1})=0\\
\\[5pt]
\bullet\;\displaystyle(-1)^{|F_1||\Upsilon|}\mathfrak{L}^{\lambda}_{X_{F_1}}(\Upsilon(X_{H_1\theta_2},X_{H_2\theta_2}))-
(-1)^{(|H_1|+1)(|F_1|+|\Upsilon|)}\mathfrak{L}^{\lambda}_{X_{H_1\theta_2}}(\Upsilon(X_{F_1},X_{H_2\theta_2}))
\\+(-1)^{(|H_2|+1)(|F_1|+|H_1|+|\Upsilon|+1)}\mathfrak{L}^{\lambda}_{X_{H_2\theta_2}}(\Upsilon(X_{F_1},X_{H_1\theta_2}))-
\Upsilon([X_{F_1},X_{H_1\theta_2}],X_{H_2\theta_2})\\+
(-1)^{(|H_1|+1)(|H_2|+1)}\Upsilon([X_{F_1},X_{H_2\theta_2}],X_{H_1\theta2})=0\\
\\[5pt]
\bullet\;\displaystyle(-1)^{(|H_1|+1)|\Upsilon|}\mathfrak{L}^{\lambda}_{X_{H_1\theta_2}}(\Upsilon(X_{H_2\theta_2},X_{H_3\theta_2}))-
(-1)^{(|H_2|+1)(|H_1|+|\Upsilon|+1)}\mathfrak{L}^{\lambda}_{X_{H_2\theta_2}}(\Upsilon(X_{H_1\theta_1},X_{H_3\theta_2}))
\\+(-1)^{(|H_3|+1)(|H_1|+|H_2|+|\Upsilon|+1)}\mathfrak{L}^{\lambda}_{X_{H_3\theta_2}}(\Upsilon(X_{H_1\theta_2},X_{H_2\theta_2}))-
\Upsilon([X_{H_1\theta_2},X_{H_2\theta_2}],X_{H_3\theta_2})\\+
(-1)^{(|H_2|+1)(|H_3|+1)}\Upsilon([X_{H_1\theta_2},X_{H_3\theta_2}],X_{H_2\theta2})-
(-1)^{(|H_1|+1)(|H_2|+|H_3|)}\Upsilon([X_{H_2\theta_2},X_{H_3\theta_2}],X_{H_1\theta2})=0,
\end{array}
\end{equation}
\vskip 0.2cm
where $X_{F_1},X_{G_1}\in\mathcal{K}(1)_1$ and $H,\;H_1,\;H_2,\;H_3\in\mathcal{H}^1$. According to (\ref{ot22}) and lemma (\ref{ot15}), we deduce the expression of $\Upsilon$. To be more precise, we get
$$\Upsilon(X_F,X_G)=\left\{
\begin{array}{lll}
\alpha\Big(F'\overline{\eta}_1\overline{\eta}_2(G)-\overline{\eta}_1\overline{\eta}_2(F)G'\Big)+
\beta\delta(-2\overline{\eta}_1\overline{\eta}_2(F'))\theta_1\theta_2 &\makebox{ if }~\lambda=0, \\[8pt]
\Big((-1)^{|F|}F'\overline{\eta}_1\overline{\eta}_2(G')-(-1)^{|G|}\overline{\eta}_1\overline{\eta}_2(F')G'\Big)\alpha_2
\\[8pt]
 (-1)^{(|F|+|G|)}\Big(\overline{\eta}_1\overline{\eta}_2(F)\overline{\eta}_1\overline{\eta}_2(G')
-\overline{\eta}_1\overline{\eta}_2(F')\overline{\eta}_1\overline{\eta}_2(G)\Big)\alpha_2&\makebox{ if }~\lambda=1, \\[8pt]
0 &\makebox{ if }~\lambda\neq0,1.
\end{array}
\right.$$

$\bullet$ For $\lambda=0.$\\
\vskip0.2cm
The 2-cocycle $\Omega_0(X_F,X_G)=\Big(F'\overline{\eta}_1\overline{\eta}_2(G)-\overline{\eta}_1\overline{\eta}_2(F)G'\Big)$ is nontrivial. Indeed, suppose that there exist a map $b$ from $\mathcal{K}(2)$ into $\mathfrak{F}^2_{\lambda}$ having the general form
\begin{equation}\label{ot23}
b(X_F)=\displaystyle\sum_{i,j\geq0}a_{ij(x,\theta)}\overline{\eta}^i_1\overline{\eta}^j_2(F)\alpha^2_{\lambda},
\end{equation}

where $a_{ij}\in\mathcal{C}^{\infty}(S^{1|2})$, such that $\Omega_0=\delta(b)$. By a direct computation, we find that the term $(f_{12}g_0'-f_0'g_{12})$ exist in the expression of $\Omega_0$ but not in the $\delta(b)$, wich gives $\Omega_0$ is nontrivial 2-cocycle.\\
\vskip0.2cm
$\bullet$ For $\lambda=1$.\\
\vskip0.2cm
For the same reason that for $\lambda=0$, we show that $\Omega_1$ and $\widetilde{\Omega}_1$ are nontrivial 2-cocycles. On the other hand $\Omega_1$ and $\widetilde{\Omega}_1$ are not cohomologous, otherwise there exist a map $b$ in the general form given by (\ref{ot23}) such that, $\alpha\Omega_1+\beta\widetilde{\Omega}_1=\delta(b)$, where $(\alpha,\;\beta)\in\mathbb{R}^2\backslash\{(0,0)\}$, but the term $\alpha(f_{12}g_0'-f_0'g_{12})+\beta(f_{12}'g_{12}-f_{12}g_{12}')$ exist in the expression of $\alpha\Omega_1+\beta\widetilde{\Omega}_1$ but not in the $\delta(b)$, which gives that $\alpha(f_{12}g_0'-f_0'g_{12})+\beta(f_{12}'g_{12}-f_{12}g_{12}')=0$, i.e.f $\alpha=\beta=0$, and then $\Omega_1$ and $\widetilde{\Omega}_1$ are not cohomologous.
\end{proof}
\subsection{Proof of Theorem \ref{ot12} }
\vskip 0.2 cm
\begin{proof}
 Consider a 2-cocycle $\Upsilon\in Z^2_{diff}(\mathcal{K}(2),\mathfrak{F}^2_{\lambda})$. If $\Upsilon/_{\mathcal{K}(1)^{\otimes2}}$ is trivial
then the 2-cocycle $\Upsilon$ is completely described by proposition \ref{ot16}. Thus, assume that  $\Upsilon/_{\mathcal{K}(1)^{\otimes2}}$ is non trivial.
Of course, by considering proposition \ref{ot14}, we deduce that nontrivial spaces $H^2_{diff}(\mathcal{K}(2),\mathfrak{F}^2_{\lambda})$ only can
appear if $\lambda\in\{-\frac{1}{2},0,\frac{1}{2},1,\frac{3}{2},2,\frac{5}{2},3,\frac{9}{2},5\}.$\\
\vskip 0.2 cm The $\mathcal{K}(1)_i$-isomorphism:
$$H^2_{diff}(\mathcal{K}(1)_i,\mathfrak{F}^2_{\lambda})\simeq H^2_{diff}(\mathcal{K}(1)_i,\mathfrak{F}^{1,i}_{\lambda})\oplus
H^2_{diff}(\mathcal{K}(1)_i,\Pi(\mathfrak{F}^{1,i}_{\lambda+\frac{1}{2}})).$$
Together with proposition \ref{ot14} describe, up to a coboundary and up to a scalar factor, the restriction of any 2-cocycle $\Upsilon$
to $\mathcal{K}(1)_1$ and to $\mathcal{K}(1)_2$. As before, we consider separately the even and the odd cases. Even cohomology spaces only can appear if
$\lambda\in\{0,1,2,3,5\}$ and odd cohomology spaces only can appear if $\lambda\in\{-\frac{1}{2},\frac{1}{2},\frac{3}{2},\frac{5}{2},\frac{9}{2}\}$.
In each case, after having the restriction of $\Upsilon$ to both $\mathcal{K}(1)_1$ and $\mathcal{K}(1)_2$, we complete the expression obtained by the corresponding other terms having the same parity of the last expression then we apply the 2-cocycles conditions we get:
$$\Upsilon(X_F,X_G)=\left\{
\begin{array}{lll}
\alpha\Upsilon^0_1(X_F,X_G)+\beta\Upsilon^0_2(X_F,X_G) &\makebox{ if }~\lambda=0, \\[8pt]
\Upsilon^3_0(X_F,X_G)&\makebox{ if }~\lambda=3, \\[8pt]
0 &\makebox{ if }~\lambda\neq0,3,
\end{array}
\right.$$
where $\Upsilon^0_1,\;\Upsilon^0_2$ and $\Upsilon^3_0$ are given by (\ref{H2013}).\\ Now, assume that there exist a map $b$ in the general form given by (\ref{ot23}) such that $\alpha\Upsilon^0_1(X_F,X_G)+\beta\Upsilon^0_2(X_F,X_G)=\delta(b)(X_F,X_G).$\\
The term $(f_{12}g_0-f_0g_{12})$ exist in the expression of $\alpha\Upsilon^0_1+\beta\Upsilon^0_2$ but not in the $\delta(b)$, wich gives that $\alpha\Upsilon^0_1+\beta\Upsilon^0_2\neq\delta(b),\;\forall(\alpha,\beta)\in\mathbb{R}^2\backslash\{(0,0)\}$. So $\Upsilon^0_1$ and $\Upsilon^0_2$ are not cohomologous. This completes the proof.

\end{proof}
\section*{Acknowledgments}

I would like to  thank Professors  \textbf{Mabrouk Ben Ammar} and \textbf{Boujemaa Agreboui} for helpful suggestions and remarks .

\def\BibTeX{{\rm B\kern-.05em{\sc i\kern-.025em b}\kern-.08em
    T\kern-.1667em\lower.7ex\hbox{E}\kern-.125emX}}


\begin{thebibliography}{99}
\small

\bibitem{bim}
Agrebaoui, B., Basdouri, I.,  Boujelben, M. (2018). The second cohomology spaces of $\mathcal{K} (1)$ with coefficients in the superspace of weighted densities and Deformations of the superspace of symbols on $\mathbb{S}^{ 1| 1}$. HAL Id: hal-01699198

\bibitem{BN}
Agrebaoui B, Ben Fraj N, On the cohomology of the Lie Superalgebra of contact vector
fields on $S^{1|1},$ {\it Belletin de la Soci\'et\'e Royale des Sciences de Li\`{e}ge}, {\bf 72}, 365?-375, (2004).


\bibitem{imnm}
Basdouri. I, Ben Ammar. M, Ben Fraj. N, Boujelbene. M, Kammoun. K, Cohomology of the Lie superalgebra of contact vector fields on
$\mathbb{K}^{1|1}$ and deformations of the superspace of symbols, {\it J. Nonlinear math. Phys. (2009)}.

\bibitem{Fuchs}
Fuchs, D B, homology of the Lie algebra of vector fields on the line, Func. Anal. Appl., \textbf{14} 201-212 (1980).

\bibitem{hnv}
Gargoubi. H, mellouli. N, Ovsienko. V, Differential operators on supercircle: conformally equivariant quantization and symbol calcuclus
, {\it Lett. Math. Phys. \textbf{79}, 51-65 (2007)}.


\end{thebibliography}
\end{document}